\documentclass[12pt]{amsart}

\usepackage{amsmath,amssymb,amsthm,verbatim}
\usepackage{hyperref, color}

\newtheorem{theorem}[subsection]{Theorem}

\newtheorem{cor}[subsection]{Corollary}

\theoremstyle{definition}

\newtheorem{remark}[subsection]{Remark}

\newcommand{\haus}{\mathcal{H}}

\newcommand{\spt}{\mathrm{spt}}

\newcommand{\reg}{\mathrm{reg}}

\newcommand{\eps}{\epsilon}

\newcommand{\sing}{\mathrm{sing}}
\newcommand{\cF}{\mathcal{F}}
\newcommand{\R}{\mathbb{R}}
\newcommand{\bC}{\mathbf{C}}
\newcommand{\C}{\mathbb{C}}
\newcommand{\N}{\mathbb{N}}

\newcommand{\del}{\partial}

\newcommand{\CP}{\mathbb{CP}}
\newcommand{\vol}{\mathrm{vol}}

\usepackage{stmaryrd}
\newcommand{\res}{\mathbin{\hspace{0.1em}\vrule height 1.3ex depth 0pt width 0.13ex\vrule height 0.13ex depth 0pt width 1.0ex}} 

\title{Entire area-minimizing surfaces in $\R^4$ are algebraic}
\author{Nick Edelen}
\address{Department of Mathematics, University of Notre Dame, Notre Dame, IN 46556 USA}
\email{nedelen@nd.edu}
\author{Luis Atzin Franco Reyna}
\address{Department of Mathematics, University of Notre Dame, Notre Dame, IN 46556 USA}
\email{lfrancor@nd.edu}
\author{Paul Minter}
\address{Department of Mathematics, Stanford University, Building 380, Stanford, CA 94305, USA}
\email{pminter@stanford.edu}

\begin{document}

\begin{abstract}
We classify entire $2$-dimensional area-minimizing or stable surfaces in $\R^4$ with quadratic area growth as algebraic, cut out by a finite union of holomorphic polynomials whose collective degrees are controlled by the density at infinity.  As a consequence, we obtain bounds on the singular set size and genus in terms of the density at infinity.
\end{abstract}
\maketitle

It is well-known that any holomorphic subvariety of $\C^n$ is area-minimizing in its homology class, and of course any (real) rotation or reflection of the variety will remain area-minimizing in $\R^{2n}$. Micallef \cite{micallef} proved a partial converse to this statement by showing that any smoothly embedded (or properly immersed) minimal $M^2 \subset \R^4$ which is stable, oriented, connected, and with quadratic area growth must be holomorphic in $\C^2$, up to a possible rigid motion in $\R^4$.  Earlier,  Morgan \cite{morgan:planes} had shown that any union of $2$-planes in $\R^n$ which is area-minimizing must be holomorphic with respect to some fixed orthogonal complex structure on their span.


We observe here that any $M$ as considered by \cite{micallef} will in fact be algebraic in $\C^2$, cut out by a holomorphic polynomial of degree equal to the density of $M$ at infinity, and with resulting bounds on the genus.  Additionally, we extend the result in \cite{micallef} to any area-minimizing $2$-current $T$ in $\R^4$ with finite density $\Theta$ at infinity, to deduce that $T$ coincides with a sum of algebraic currents in $\C^2$, with a priori bounds on the degree, singular set, and genus in terms of $\Theta$.  The constituent algebraic currents of $T$ may be holomorphic with respect to different orthogonal complex structures, which slightly complicates the singular set bound, and though there is some compatibility between the complex structures it is an interesting question whether they all must coincide.

Algebraicity of $M$ is perhaps already known to some experts (as it can be deduced from holomorphicity and the area growth assumption from complex-analytic methods such as \cite[Theorem D]{stolzenberg}), but it is not mentioned in \cite{micallef} or anywhere else we are aware of.  We give a more minimal surface oriented proof, using the uniqueness of tangent cones at infinity due to Rivière \cite{riviere} to show that $\overline{M}$ extends to an analytic subvariety of $\CP^2$, and then invoking the principle of ``GAGA.''

\vspace{3mm}

Given a stationary $2$-current $T$ in $\R^4$, write $$\Theta_T(\infty) := \lim_{r \to \infty} \frac{\|T\|(B_r(0))}{\pi r^2}$$ for the density at infinity.  It follows by the monotonicity formula that $\Theta_T(\infty)$ always exists (although a priori it could be infinite).  Write $\reg\, T \subset \spt\, T$ for the set of points near which $\spt\, T$ coincides with a smooth, embedded submanifold, and write $\sing\, T = \spt\, T \setminus \reg\, T$.  For a subset $A \subset \R^4$ write $\overline{A}$ for its closure.

If $M$ is an immersed or embedded oriented $2$-surface in $\R^4$, we can define $\Theta_M$, $\reg\, M$, $\sing\, M$ in terms of its associated $2$-current $\llbracket M \rrbracket$, so that e.g.~$\|M\|(A) \equiv \haus^2(M \cap A)$ and $\sing\, M$ consists of the points in $\overline{M}$ about which $\overline{M}$ does not locally coincide with a smooth embedded submanifold.

Our first theorem is as follows.

\begin{theorem}\label{thm:M}
Let $F:M^2\to \R^4\equiv \C^2$ be a connected orientable smooth stable minimal immersion without boundary, which:
\begin{enumerate}
    \item [\textnormal{(i)}] is proper away from a discrete set $X\subset\R^4$, i.e.~$F^{-1}(K)$ is compact for every compact $K\subset\R^4\setminus X$;
    \item [\textnormal{(ii)}] has quadratic area growth with respect to ambient balls, i.e.~$\Theta:=\Theta_{F(M)}(\infty)<\infty$.\footnote{Our hypotheses imply that $F(M)$ defines a stationary integral varifold in all of $\R^4$, so by monotonicity the limit $\Theta_{F(M)}(\infty)$ 
    will exist in $[1,\infty]$ (cf. Remark \ref{rem:M-to-V}).}
\end{enumerate}
Then, we can write
\begin{equation}\label{eqn:M-concl1}
\overline{F(M)} = A\{(z,w)\in \C^2:p(z,w) = 0\}
\end{equation}
for some $A \in O(4)$, and $p:\C^2\to \C$ a holomorphic polynomial of degree $\Theta$. In particular, $F(M)$ is smooth and properly embedded away from a finite singular set, with bounds:
\begin{equation}\label{eqn:M-concl2}
\#\sing\, F(M) + \mathrm{genus}(\reg\, F(M)) \leq \frac{1}{2} (\Theta - 1)(\Theta - 2).
\end{equation}
\end{theorem}

\begin{remark}
As is well-known, holomorphicity implies the $2$-current $\llbracket F(M) \rrbracket$ is area-minimizing, and also implies that $M$ can only be immersed at isolated points.  We remark that a pair of planes meeting along a line is a stable minimal surface, but disconnected as an immersion, so Theorem \ref{thm:M} (or Micallef's holomorphicity result \cite{micallef}) would only apply to each connected component of the immersion.
\end{remark}


\begin{remark}
In contrast to $\R^{2n}\equiv\C^n$, stable minimal surfaces in $\CP^n$ are vastly more rigid: Lawson--Simons \cite{lawson-simons} proved that any integral $p$-current in $\CP^n$ is stable if and only if it is a complex-algebraic cycle.
\end{remark}

\begin{remark}
We highlight Theorem \ref{thm:M} applies to $M^2 \subset \R^4$ that are connected, oriented, smoothly embedded, stable minimal $2$-surfaces which are complete away from a discrete set (i.e.~$\overline{M} \setminus M$ is discrete in $\R^4$)\footnote{Strictly speaking, for a general submanifold being complete is not the same as being set-theoretically closed.  However, since $M$ has finite area in any ambient ball it follows by e.g.~the curvature estimates of \cite{choi-schoen} that $\overline{M} = M$ in the region where $M$ is complete.} and have quadratic area growth (i.e.~$\Theta_M(\infty) < \infty$).
\end{remark}

Our second main theorem is the following.
\begin{theorem}\label{thm:main}
Let $T$ be an area-minimizing integral $2$-current in $\R^4\equiv \C^2$ with $\del T = 0$ and $\Theta:=\Theta_T(\infty)<\infty$. Then, there exist positive integers $m_1,\dotsc,m_k$, holomorphic polynomials $\{p_i:\C^2\to\C\}_{i=1}^k$ with $d_i:=\deg(p_i)\geq 1$ satisfying $m_1 d_1 + \cdots + m_k d_k = \Theta$, $\ell \in \{ 1, 2, \ldots, k\}$, complex rigid motions $\{ B_i \in U(2) \}_{i=\ell+1}^k$, and a real rigid motion $A \in O(4)$ so that
\begin{align}
T &= A_\# \sum_{i=1}^\ell m_i \llbracket \{(z, w)\in \C^2 : p_i(z, w) = 0 \} \rrbracket \label{eqn:main-concl1} \\
&\qquad\qquad + A_\# \sum_{i=\ell+1}^k m_i {(B_i R)}_\# \llbracket \{ (z, w)\in \C^2 : p_i(z, w) = 0 \}\rrbracket . \nonumber
\end{align}
Here $R(z, w) = (z, \overline{w})$.  Either $d_i = 1$ for every $i$ in which case we can assume $\ell = k$; or we can assume at least one degree $\{ d_i : i \leq \ell \}$ is at least $2$, every degree $\{ d_i : i \geq \ell+1\}$ is at least $2$, and 
\[
p_i(z, w) = w^{d_i} + \textnormal{(lower order terms in $z, w$)} \quad \forall i \geq \ell+1.
\]

Additionally, we have the bounds
\begin{align}\label{eqn:main-concl2}
&\# \sing\, T  \leq \Theta^3, 
\end{align}
and
\begin{equation}\label{eqn:main-concl3}
\textnormal{genus}(\reg\,T) \leq \frac{1}{2}\sum_i(d_i-1)(d_i-2) \leq \frac{\Theta^2}{2}
\end{equation}
\end{theorem}

\begin{remark}
A slightly weaker formulation of \eqref{eqn:main-concl1} could simply be
\[
T = A_\# \sum^k_{i=1}m_i {(A_i)}_\# \llbracket \{(z,w)\in \C^2:p_i(z,w)=0\}\rrbracket
\]
for suitable $A_i \in O(4)$, but \eqref{eqn:main-concl1} captures some compatibility of the complex structures between the different $T_i := m_i {(A_i)}_\# \llbracket \{(z,w)\in \C^2:p_i(z,w)=0\}\rrbracket$, which is enforced by the condition that any tangent cone of $T$ must consist of planes all positively oriented and holomorphic with respect to the same orthogonal complex structure (\cite{morgan:planes}).

In particular, \eqref{eqn:main-concl1} says that each $T_i$ is either positively holomorphic with respect to the standard complex structure $J$, or there is a (positively oriented) complex plane $P_i \equiv B_i \{ w = 0 \}$ so that $T_i$ is positively holomorphic with respect to $(R_i)_\# J$ for some real reflection $R_i$ in $P_i^\perp$.  This second case is somewhat restrictive: it can only occur if the tangent cone of $T_i$ at infinity is the plane $m_i d_i \llbracket P_i \rrbracket$, and moreover given any $j \leq \ell < i$ then $\spt\, T_i$, $\spt\, T_j$ can only intersect at points $x$ where the tangent cone of either $T_i$ or $T_j$ at $x$ is supported in $P_i$.

It may be that all the $T_i$ are holomorphic with respect to the same orthogonal complex structure (i.e.~$\ell = k$ in \eqref{eqn:main-concl1}), so that $T$ will simply be a positive $(1, 1)$-current up to a real rigid motion, and which would also yield the (possibly sharp) bound \begin{equation}\label{eqn:rem-sharp}
\#\sing\,T+\textnormal{genus}(\reg\,T) \leq \frac{1}{2}\sum_{i\neq j}d_id_j + \frac{1}{2}\sum_i(d_i-1)(d_i-2) \leq \frac{\Theta^2}{2}.
\end{equation}
If $\Theta_T(\infty) \leq 3$ this is true because of the restrictions Theorem \ref{thm:main} puts on the degrees $d_i$.  However, it's not immediately clear to us why, e.g.~a current of the form
\[
\llbracket \{ w^2 = z \} \rrbracket + R_\# \llbracket \{ w^2 = z + 1 \} \rrbracket
\]
could not be area-minimizing (note the tangent cone at infinity of this current is $4\llbracket \{ w = 0 \} \rrbracket$).
\end{remark}

\begin{remark}
The estimate \eqref{eqn:main-concl2} is not optimal.  We suspect $\# \sing T$ should be bounded quadatrically in terms of the density $\Theta_T(\infty)$, and it's possible that \eqref{eqn:rem-sharp} is the sharp estimate even when $\ell < k$.  The cubic nature of our bound \eqref{eqn:main-concl2} arises from the intersections between $T_i$, $T_j$ for $i \leq \ell < j$, but these feel highly non-generic (as any such intersection must be the zero set of 6 real polynomial equations).
\end{remark}

\begin{remark}
The assumption $\Theta_T(\infty)<\infty$ is obviously necessary, as the graph of any holomorphic function $f:\C\to \C$ is area-minimizing in $\C^2$, but need not have quadratic volume growth (e.g. $f(z) = e^z$).
\end{remark}

For low densities like $\Theta_T(\infty) = 2, 3$, Theorem \ref{thm:main} gives a more rigid characterization of area-minimizers.  For example, we have:
\begin{cor}\label{cor:main}
Let $T$ be an area-minimizing integral $2$-current in $\R^4\equiv \C^2$ with $\del T = 0$ and $\Theta_T(\infty)=2$. Then for some $A \in O(4)$, and some $a, b, c, d, e, f, \in \C$, $T$ is one of two possibilities:
\begin{enumerate}
    \item [\textnormal{(i)}] the sum of two (possibly coincident) positively oriented, complex affine planes:
    \begin{align*}
    T &= A_\#\llbracket \{(z, w) \in \C^2 : az + bw = c \rrbracket \\
    &\quad \quad + A_\#\llbracket (z, w) \in \C^2 : d z + ew = f \rrbracket;
    \end{align*}
    \item [\textnormal{(ii)}] entirely regular, multiplicity one, genus zero, and the zero set of a quadratic holomorphic polynomial:
    \begin{align*}
    T = A_\# \llbracket (z, w) \in \C^2 : az^2 + bzw + cw^2 + dz + ew + f = 0 \rrbracket .
    \end{align*}
\end{enumerate}
\end{cor}

We thank Jeff Diller, Christos Mantoulidis, Davide Parise, Nick Salter, and Eric Riedl for helpful conversations, and Brian White for his interest and encouragement. N.E. was supported by NSF grant DMS-2506700 and a Simons Foundation travel award. This research was conducted during the period P.M. served as a Clay Research Fellow.

\section{Proof of Main Result}

We identify $\R^4 \equiv \C^2 \equiv \{ (z, w) : z, w \in \C \}$.  We will typically write $J$ for the standard complex structure on $\C^2$, and given $A \in O(4)$ write $A_\#J \equiv A J A^{-1}$ for the orthogonal complex structure induced by $A$.  Given a subset $A \subseteq \C^2$ (or $\CP^2$), $\overline{A}$ will always denote the set-theoretic closure.  Write $\haus^k$ for the $k$-dimensional Hausdorff measure. We will use the language of integral currents and varifolds, and refer the reader to e.g.~\cite{simon:gmt} for background and notation. We write $\cF$ for the flat norm of a compactly supported current. Given an immersed $2$-surface $M$, we write $|M|$ for the induced varifold, and (if $M$ is oriented) $\llbracket M\rrbracket$ for the induced current, so that $\|M\|(A) \equiv \haus^2(M \cap A)$.

Given a $2$-current $T$, write
\[
\Theta_T(x, r) = \frac{\|T\|(B_r(x))}{\pi r^2}
\]
for the density of $T$ in $B_r(x)$.  If $T$ is stationary, the monotonicity formula asserts that $\Theta_T(x, r)$ is increasing in $r$, and so the densities $\Theta_T(x) := \lim_{r \to 0} \Theta_T(x, r)$ and $\Theta_T(\infty) := \lim_{r \to \infty} \Theta_T(0, r)$ always exist (the latter possibly being $+\infty$).

Recall the well-known fact that any conical area-minimizing integral $2$-current in $\R^m$ (of finite mass) consists of a finite union of $2$-planes (possibly with multiplicity) which are disjoint away from $0$.  In particular, if $T$ is an area-minimizing integral $2$-current, then $\Theta_T(x) \in \N$ for every $x\in \spt\,T$, and $\Theta_T(\infty) \in \N\cup\{\infty\}$.

The first result we need is uniqueness of tangent cones as infinity.

\begin{theorem}[Uniqueness of tangent cones at infinity, \cite{riviere}]\label{thm:unique}
Let $T^2$ be an area-minimizing integral $2$-current in $\R^{2+m}$ with $\del T = 0$ and $\Theta := \Theta_T(\infty) < \infty$. Then, there are oriented $2$-planes $P_1,\dotsc,P_\ell$ which are disjoint away from $0$ and multiplicities $n_1,\dotsc,n_\ell\in\N$ such that the cone $\bC = n_1\llbracket P_1\rrbracket + \cdots + n_\ell\llbracket P_\ell\rrbracket$ is the unique tangent cone to $T$ at infinity. In other words, as $R\to\infty$ we have
\begin{enumerate}
    \item [\textnormal{(i)}] $(1/R)_\#T\to \bC$ as currents and as varifolds;
    \item [\textnormal{(ii)}] $(1/R)\spt\,T\to \spt\,\bC$ locally in Hausdorff distance.
\end{enumerate}
\end{theorem}

\begin{proof}
The uniqueness of $\bC$ follows from the reverse epiperimetric inequality established in \cite{riviere}.  More precisely, \cite{riviere} proves that there is an $\eps(m, \Theta) > 0$ such that whenever there are oriented $2$-planes $P_1, \ldots, P_{\Theta}$ (any two being either coincident or transverse) satisfying
\begin{equation}\label{eqn:unique-1}
\cF\left(T\res B_2 - \sum^\Theta_{i=1}\llbracket P_i\rrbracket \res B_2\right) < \eps
\end{equation}
then it follows that
\[
\pi \Theta - \frac{1}{2} \|\del(T \res B_1)\|(\del B_1) \leq (1-\eps)(\pi \Theta - \|T\|(B_1)).
\]

A straightforward argument by contradiction implies that \eqref{eqn:unique-1} holds with $(1/R)_\# T$ in place of $T$ for all $R \gg 1$. Hence, \cite{riviere} obtains a decay estimate of the form
\[
\Theta - \Theta_T(0, R) \leq \left(\frac{R_0}{R}\right)^{2\eps}(\Theta - \Theta_T(0, R_0)) \quad \text{for all } R \geq R_0
\]
for some $R_0\gg 1$. By an argument similar to that in \cite{white:unique}, this implies that given any sequence $R_i \to \infty$, the currents $\del (((1/R_i)_\# T) \res B_1)$ form a Cauchy sequence in the flat norm, which implies there is a unique tangent cone $\bC$ of $T$ at infinity. From this all the claimed convergence results follow. 
\end{proof}

Next, we lay out the key holomorphicity result due to \cite{micallef} 

\begin{theorem}[Holomorphicity of stable minimal surfaces in $\R^4$, \cite{micallef}]\label{thm:holo}
Let $F : M^2 \to \R^4 \equiv \C^2$ be a smooth, isometric, stable minimal immersion of a connected, oriented $2$-surface without boundary, obeying:
\begin{enumerate}
    \item [\textnormal{(i)}] $F$ is proper away from a discrete set $X\subset\R^4$ (i.e.~$F^{-1}(K)$ is compact for every compact $K\subset\R^4\setminus X$);
    \item [\textnormal{(ii)}] $F(M)$ has quadratic area growth with respect to ambient balls (i.e.~$\Theta := \Theta_{F(M)}(\infty)<\infty$).
\end{enumerate}
Then, we can find a (real) rigid motion $A\in O(4)$ so that $A\circ F$ is holomorphic with respect to the standard complex structure on $\C^2$.
\end{theorem}

\begin{remark}\label{rem:M-to-V}
By properness $|F(M)|$ defines a stationary integral varifold in $\R^4 \setminus X$, and since $X$ is discrete and $|F(M)|$ has finite mass in any ambient ball, $|F(M)|$ extends as a stationary integral varifold to all of $\R^4$. Therefore by the monotonicity formula we have $\Theta_{F(M)}(x, r) \leq \Theta$ for all $x \in \R^4$ and all $r > 0$.  
\end{remark}

\begin{proof}
As mentioned, \cite{micallef} originally proved Theorem \ref{thm:holo} in the case when $X = \emptyset$.  A minor modification of the proof in \cite{micallef}, essentially outlined in \cite{micallef:note} in the setting of branched immersions, can adapt the result to the case of a discrete singular set (and $M$ which is not a priori complete). There are two necessary facts to verify.  For convenience there is no loss below in replacing $M$ with $M \setminus F^{-1}(X)$, since $F^{-1}(X)$ is discrete (possibly empty or infinite).

First, we require the following result of \cite{fischer-colbrie}, originally proved for complete, non-compact $M$: if for some smooth function $Q$ on $M$ we have
\[
\int_M Q f^2 \leq \int_M |\nabla f|^2 \quad \forall f \in C^1_c(M)
\]
then there is a positive function $u$ on $M$ satisfying $\Delta u + Q u = 0$. In \cite{fischer-colbrie} this is argued by solving a sequence of PDEs on a compact exhaustion $D_i \equiv B_{R_i}(p)\subset M$
of $M$, renormalizing, and then using the Harnack inequality to take a limit to obtain the desired $u$.  It is readily verifiable that when $F: M^2 \to \R^4$ is non-compact and properly immersed away from a discrete singular set $X= \{ x_j \}_{j} \subset \R^4$, then the same argument works with domains of the form $D_i = F^{-1}(B_{R_i} \setminus \cup_j B_{r_{ij}}(x_j))$ for suitable $R_i \to \infty$, $r_{ij} \to 0$.

Second, we require that $M$ is parabolic, i.e.~that any positive superharmonic function on $M$ is constant.  For this we employ the log cutoff trick.  We first note that by Remark \ref{rem:M-to-V} we have the bounds $\haus^2(F(M) \cap B_r(x)) \leq \Theta \pi r^2$ for all $x \in \R^4$ and $r > 0$.  We first claim that there is a number $k$ such that $F^{-1}(F(x)) \leq k$ for $\haus^2$-a.e. $x \in M$.

To see this, consider the set $M_k \subseteq M$ with the property that for every $x \in M_k$, there are $k$ disjoint open subsets $V_1, \ldots, V_k \subset M$ such that $x \in V_1$ and $F(V_1) = F(V_2) = \cdots = F(V_k)$.  By (real) analytic continuation, $M_k$ is both open and closed in $M$, so either $M_k = \emptyset$ or $M_k = M$.  Trivially $M_1 = M$.  Let us fix $k$ to be the largest integer for which $M_k = M$, which is finite by the properness of $F$.

Given $p \in F(M) \setminus X$ and $r > 0$ sufficiently small, we can write $F(M) \cap B_r(p) = M_1 \cup \cdots \cup M_l$ for each $M_i$ the embedding of a small disk and $l < \infty$.  For each $i \neq j$, either $M_i, M_j$ coincide near $p$, or by \cite{hoffman-osserman} (see also \cite{micallef-white}) the set of points at which $M_i, M_j$ intersect tangentially is discrete. On the other hand, by the inverse function theorem, the set of points at which $M_i, M_j$ intersect non-tangentially has Hausdorff dimension $\leq 1$.  Since $\{ x \in F^{-1}(B_r(p)) : \# F^{-1}(F(x)) > k \}$ is contained in the set $F^{-1}( \cup_{i, j} \{ M_i \cap M_j : \text{$M_i, M_j$ do not locally coincide}\})$, we deduce $\{ x \in M : \# F^{-1}(F(x)) > k \}$ has $\haus^2$-measure zero.  This shows $\#F^{-1}(F(x)) = k$ for $\haus^2$-a.e. $x \in M$.

Following e.g.~\cite[Example 3.10(3)]{christos:notes}, to show $M$ is parabolic it suffices to prove there is a family of non-negative functions $\psi_i \in C^{0, 1}_c(M)$ so that $\psi_i \to 1$ uniformly on compact subsets of $M$, and $\int_M |\nabla \psi_i|^2 d\haus^2 \to 0$.  To this end, fix $\eps > 0$, and recalling the enumeration $X = \{ x_j \}_{j}$ set $r_j = \exp(-2^i/\eps)$ for $j\geq 1$ and $r_0 = \exp(1/\eps)$. Define
\[
\phi_j(y) = \left\{ \begin{array}{l l} 0 & |y - x_j| < r_j^2 \\ \frac{\log(|y - x_j|) - \log(r_j^2)}{\log(r_j) - \log(r_j^2)} & r_j^2 \leq |y - x_j| \leq r_j\\ 1 & r_j < |y - x_j| \end{array} \right. \quad \text{for }j\geq 1
\]
and
\[
\phi_0(y) = \left \{ \begin{array}{l l} 1 & |y| < r_0 \\ \frac{ \log(r^2_0)-\log(|y|)}{\log(r_0^2) -\log(r_0)} & r_0 \leq |y| \leq r_0^2 \\ 0 & |y|>r_0^2 \end{array} \right. .
\]
Set $\phi(y) = \min_{j \geq 0} \phi_j(y)$, which by our discreteness hypothesis is actually only a minimum over finitely many $j$ locally about each point. Thus $\phi \in C^{0, 1}_c(\R^4)$ and clearly $\phi \to 1$ uniformly on compact sets of $\R^4 \setminus X$ as $\eps \to 0$.

Define $\psi = \phi \circ F \in C^{0, 1}_c(M)$.  Then $\psi \to 1$ in $C^0_{\text{loc}}(M)$ as $\eps \to 0$, and we compute
\begin{align*}
\int_M |\nabla \psi|^2 d\haus^2 
&\leq k \int_{F(M) \cap B_{r_0^2} \setminus B_{r_0}} \frac{1}{|y|^2 \log^2(r_0)} d\haus^2 \\
&\quad \quad\quad + k \sum_{j\geq 1} \int_{F(M) \cap B_{r_j}(x_j) \setminus B_{r_j^2}(x_j)} \frac{1}{|y - x_j|^2 \log^2(r_j)} d\haus^2 \\
&\leq \frac{16 \Theta k\pi}{\log(r_0)} + \sum_{j\geq 1} \frac{16 \Theta k \pi}{-\log(r_j)} \\
&\leq 32 \Theta k \pi \eps,
\end{align*}
having used the bound $\# F^{-1}(F(x)) \leq k$ in the first inequality, and the bounds $\haus^2(F(M) \cap B_r(x)) \leq \Theta \pi r^2$ (over suitable dyadic scales) in the second.

\vspace{3mm}

The rest of the proof of \cite{micallef} carries over without change (see also \cite{christos:notes}), however we sketch the general argument.

The starting point, since $M$ is orientable, is to endow both $M$ and the (pullback) normal bundle $NM \subset F^*(T\R^4) \cong M \times \R$ with complex structures $J^\top$ and $J^\perp$ (respectively) by rotation by $90^\circ$. Define for each $p \in M$ the mapping $J_p(v) = J^\top(\pi_{T_pM}(v)) + J^\perp(\pi_{N_pM}(v))$, which gives an orthogonal complex structure on $T_{F(p)}\R^4$.  The end goal is demonstrate that $J_p$ is independent of $p$.

Write $(z, \bar z)$ for (local) complex coordinates on $M$.  Split the complexified tangent and normal bundles via the complex structures $J^\top$, $J^\perp$ as $T_\bC M = T^{1, 0}M \oplus T^{0, 1}M$, $N_\bC M = N^{1, 0}M \oplus N^{0, 1}M$.  Write $F_{zz}^{1, 0} = \pi_{N^{1, 0}M}(F_{zz})$, $F_{zz}^{0, 1} = \pi_{N^{0, 1}M}(F_{zz})$.

A pointwise computation shows that if either $F_{zz}^{1, 0}$ or $F_{zz}^{0, 1}$ identically vanishes on $M$, then after possibly replacing $J^\perp$ with $-J^\perp$, the mapping $J_p$ will be independent of $p$.  To this end, \cite{micallef} uses the second variation formula to get the inquality
\[
\int_M q(a) f^2 \leq |a| \int_M |\nabla f|^2 \qquad \forall f \in C^1_c(M),
\]
for any $a \in \C^4$ fixed, where $q(a) = \frac{-2}{|F_z|^4} \mathrm{Re}\left\{ (F_{zz}^{1, 0}\cdot a)(F_{zz}^{0, 1}\cdot \overline{a}) \right\}$.

Taking $a \in \{ e_i \}_{i = 1}^4$ among an orthonormal basis of $\C^4$, the (modified) argument of \cite{fischer-colbrie} implies there are smooth, positive $u_i$ on $M$ solving $\Delta u_i + q(e_i) u_i = 0$.  \cite{micallef} then observes that, since $\sum_{i=1}^4 q(e_i) = 0$, if one sets $w_i = \log(u_i)$ we obtain
\[
\frac{1}{16} \int_M \Big(\sum_{i=1}^4 |\nabla w_i|^2\Big) f^2 \leq \int_M |\nabla f|^2 \qquad \forall f \in C^1_c(M).
\]
Applying the modified \cite{fischer-colbrie} again we deduce there is a smooth, positive $f$ on $M$ satisfying $\Delta f + \frac{1}{16}(\sum_{i=1}^4 |\nabla w_i|^2) f = 0$, and in particular $\Delta f \leq 0$. By parabolicity of $M$, $f$ must be constant, and therefore all the $w_i$ are constant, and hence $|F_{zz}^{1, 0}| |F_{zz}^{0, 1}| \equiv 0$.  Since both $F_{zz}^{1, 0} dz^2$, $F_{zz}^{0, 1} dz^2$ are holomorphic differentials, it follows that one of the two factors must identically vanish on $M$.
\end{proof}

We now prove the main theorems.
\begin{proof}[Proof of Theorem \ref{thm:M}]
We first use Theorem \ref{thm:holo} to get that, possibly after a suitable rotation and reflection, $F : M^2 \to \C^2$ is holomorphic, so that $F(M)$ is a closed, complex-analytic subset of $\C^2 \setminus X$.  Since $X$ is discrete, standard removable singularity theorems (e.g.~\cite[Theorem 8.7]{demailly}) imply $\overline{F(M)}$ is a complex-analytic subset of $\C^2$.  Outside of a discrete singular set $\overline{F(M)}$ is properly embedded, so henceforth let us identify $M$ with $\reg\, F(M)$.

The current $T = \llbracket M\rrbracket$ is a closed, positive, integral $(1, 1)$-current in $\C^2$ (c.f.~Remark \ref{rem:M-to-V}), and so in particular $T$ is area-minimizing.  By Theorem \ref{thm:unique}, $T$ has a unique tangent cone $\bC = n_1 \llbracket P_1\rrbracket + \cdots + n_l\llbracket P_l\rrbracket$ at infinity (attained as currents, varifolds, and in Hausdorff distance for the supports), where without loss of generality we can assume $P_i \cap P_j = \{ 0 \}$ for all $i\neq j$.  From holomorphicity of $T$ each $P_i$ is a complex subspace of $\C^2$ (see also \cite{morgan:planes}).

Embed $\C^2$ into $\CP^2$, and write $\CP^1$ for the sphere at infinity.  We can think of $M$ as an analytic subset of $\CP^2 \setminus \CP^1$, which by the previous paragraph satisfies
\[
\overline{M} \cap \CP^1 = (\cup_i P_i) \cap \CP^1 
\]
is a discrete set.  So in fact $\overline{M} \cap \C^2$ is an analytic subset of $\CP^2$ minus a finite collection of points, and therefore again using the removable singularity results $\overline{M}$ is an analytic subset of $\CP^2$.  From Chow's theorem (\cite[Theorem 8.10]{demailly}) it follows that $\overline{M}$ is algebraic in $\CP^2$, and thus there is a homogeneous polynomial $P(z, w, y)$ on $\C^3$ so that $\overline{M} = \{ [z : w : y] : P(z, w, y) = 0 \} \subset \CP^2$.  Setting $p(z, w) = P(z, w, 1)$, we get $\overline{M} \cap \C^2 = \{ (z, w) : p(z, w) = 0 \}$.

Since $M$ is connected, there is no loss in assuming $P(z, w, y)$ is irreducible.  A basic fact about the cohomology of algebraic curves in $\CP^2$ (e.g. \cite[Theorem 5.22]{mumford1}) implies 
\[
\deg(P) \pi = \vol_{FS}(M) = \int_M \omega_{FS}
\]
where $\vol_{FS}$ denotes the volume of $M \subset \CP^2$ with respect to the Fubini-Study metric, and $\omega_{FS}$ denotes the Fubini-Study Kahler form on $\C^2$, i.e.~writing $(z_1, z_2) = (z, w)$, then
\[
\omega_{FS} = \frac{\sqrt{-1}}{2} \sum_{i, j = 1}^2 \frac{(1+|z_1|^2 + |z_2|^2) \delta_{i \bar j} - \bar{z_i} z_j}{(1+|z_1|^2 + |z_2|^2)^2} dz^i \wedge d \bar z^j
\]
On the other hand, by direct computation we have
\[
\int_{\bC} \omega_{FS} = \sum_{i=1}^l n_i \int_{P_i} \omega_{FS} = \Theta \pi
\]
By Theorem \ref{thm:unique}, we can choose a sequence of radii $R_i \to \infty$, $3$-currents $E_i$ supported in $\overline{B_{R_i}}$, and $2$-currents $S_i$ supported in $\del B_{R_i}$, so that 
\[
(\llbracket M\rrbracket - \bC) \res B_{R_i} = \del E_i + S_i, \quad R_i^{-2} \|S_i\|(\del B_{R_i}) \to 0.
\]
Now since $\omega_{FS}$ is closed, and noting that $|\omega_{FS}| \leq c R^{-2}$ on $\del B_R$ for a uniform constant $c$, we have
\[
\left| \int_{M \cap B_{R_i}} \omega_{FS} - \int_{\bC \cap B_{R_i}} \omega_{FS} \right| \leq c R_i^{-2} \|S_i\|(\del B_{R_i}) \to 0
\]
as $i \to \infty$.  Since $\omega_{FS}$ is the volume form (and hence $\geq 0$) on both $M$ and $\bC$, and both integrals are finite, we deduce that $\int_M \omega_{FS} = \int_\bC \omega_{FS}$, and hence $\deg(P) = \Theta$.

(A more hands-on analysis would show $p(z, w) = (a_1 z + b_1 w)^{n_1} \cdots (a_l z + b_l w)^{n_l} + q(z, w)$ where $P_i = \{ (z, w) : a_i z + b_i w = 0 \}$ and $\deg(q) \leq d-1$.)

Finally, the singular set and genus bound follows from the genus-degree formula (e.g.~\cite[Chapter IV, Exercise 1.8]{hartshorne}).
\end{proof}

\begin{proof}[Proof of Theorem \ref{thm:main}]
By \cite{chang} (see also \cite{delellis-spadaro-spolaor3}
), $T$ is regular away from a discrete singular set $\sing\,T$.  Write $M_1, \ldots, M_k$ for the connected components of $\reg\, T$, so that if we set $T_i = T \res M_i$ then (from the constancy theorem and integrality of $T$) $T_i = m_i\llbracket M_i\rrbracket$ for some positive integer $m_i$ and
\[
T = T_1 + \cdots + T_k, \quad \|T\| = \|T_1\| + \cdots + \|T_k\|.
\]
In particular, we have $m_1 \Theta_{M_1}(\infty) + \cdots + m_k \Theta_{M_k}(\infty) = \Theta_T(\infty)$.  Note that since each $M_i$ is itself stationary, monotonicity implies $\Theta_{M_i}(\infty) \geq 1$, and so $k < \infty$.

Each $M_i$ is stable, connected, oriented, embedded, and complete away from the discrete set $\sing\, M$, and has $\Theta_{M_i}(\infty) < \infty$.  We can therefore apply Theorem \ref{thm:M} to deduce
$$M_i = A_i \{ (z, w)\in \C^2 : p_i(z, w) = 0 \}$$
for $A_i \in O(4)$, and $p_i(z, w)$ a degree $d_i \equiv \Theta_{M_i}(\infty)$ polynomial.

Let $\bC = \llbracket P_1\rrbracket + \cdots + \llbracket P_\Theta\rrbracket$ be the (unique) tangent cone of $T$ at infinity.  By \cite{morgan:planes}, all the $P_i$ must be holomorphic and positively oriented with respect to the same orthogonal complex structure, which after a rigid motion in $\R^4$ we can assume is $J$.  Given any $T_i$, there is some subset $I_i \subseteq \{ 1, \ldots, \Theta\}$ so that the tangent cone at infinity for $T_i$ is $\bC_i = \sum_{j \in I_i} \llbracket P_j\rrbracket$.

Each $P_j$ ($j \in I_i$) must be holomorphic with respect to both orthogonal complex structures $J$ and $({A_i})_\# J$.  If the $\{ P_j : j \in I_i \}$ span $\R^4$, then the complex structure is uniquely determined, and we must have $J = ({A_i})_\# J$, hence $A_i \in U(2)$.  If all the $\{ P_j : j \in I_i \}$ coincide, then the complex structure is uniquely determined up to the choice of (90$^\circ$) rotation in $P_j^\perp$, and so we have either $J = ({A_i})_\# J$ or $({R_i})_\# J = ({A_i})_\# J$ where $R_i$ is any choice of real reflection in $P_j^\perp$, hence either $A_i \in U(2)$ or $R_i^{-1} A_i \in U(2)$.  In this last case, since $P_j$ is complex, we can write $R_i = B_i R B_i^{-1}$ where $B_i \in U(2)$ and $R(z, w) = (z, \overline{w})$.

Therefore, after changing each polynomial $p_i(z, w)$ by a complex rigid motion, we can assume that either
\[
T_i = m_i \llbracket \{ p_i(z, w) = 0\} \rrbracket, \quad \text{ or } \quad T_i = m_i {(B_i R)}_\# \llbracket \{ p_i(z, w) = 0 \} \rrbracket .
\]
In the second case above we must have that $\bC_i$ is supported in the plane $B_i\{ w = 0 \}$, and so $p_i(z, w) = w^{d_i} + \text{(lower order terms in $z, w$)}$ and $\bC_i = m_i d_i (B_i)_\# \llbracket \{ w = 0 \} \rrbracket$. Since $\bC$ must be positively oriented with respect to the standard complex structure, we can keep $m_i > 0$ if we endow each current $\llbracket \{ p_i(z, w) = 0 \} \rrbracket$ with the positive orientation.  This gives our decomposition \eqref{eqn:main-concl1}.  The bound \eqref{eqn:main-concl3} on genus follows directly from Theorem \ref{thm:M}.

\vspace{3mm}

We now prove the bound \eqref{eqn:main-concl2} on the singular set. Any singular point of $T$ can arise as a point in some $\sing\, T_i$ or as a point in some intersection $\spt\, T_i \cap \spt\, T_j$ (or even just $\reg\, T_i \cap \reg\, T_j$).  We break into three cases:
\begin{enumerate}
\item both $T_i, T_j$ are holomorphic with respect to the standard complex structure $J$ (so $i, j \leq \ell$); 
\item  $T_i$ is $J_i$-holomorphic, and $T_j$ is $J_i$-holomorphic for $J_i, J_j$ both orientation-reversing orthogonal complex structures (so, $i, j \geq \ell+1$); 
\item $T_i$ is $J$-holomorphic and $T_j$ is $J_j$-holomorphic for some orientation reversing-complex structure $J_j$ (so $i \leq \ell < j$).
\end{enumerate}

\emph{Case (1).} In this case we can simply use Bezout to bound
\begin{equation}\label{eqn:main-1}
\# \spt\, T_i \cap \spt\, T_j = \# \{ (z, w) : p_i(z, w) = p_j(z, w) = 0 \} \leq d_i d_j.
\end{equation}

\emph{Case (2).} We will show that the bound \eqref{eqn:main-1} holds in Case (2) also.  The naive bound using ``real'' Bezout in this case will be $d_i^2 d_j^2$, however we can use the minimality of $T_i + T_j$ to deduce (after a suitable perturbation) that $\spt\, T_i \cap \spt\, T_j = \reg\, T_i \cap \reg\, T_j$ transversely and with positive intersection number, which will allow us to apply ``standard'' Bezout (we stress that area-minimality of the sum of crucial for this).

Of course if $J_1 = J_2$ then up to (real) rigid motion both $T_i$, $T_j$ are $J$-holomorphic, and \eqref{eqn:main-1} again follows from Bezout, so we can without loss of generality assume $J_1 \neq J_2$.

For ease of notation, in this case we will simply write $i = 1$, $j = 2$, and let $k \in \{ 1, 2 \}$.  Our setup is
\begin{gather*}
T_k = m_k (B_k R)_\# \llbracket \{ p_k(z, w) = 0 \} \rrbracket, \quad  B_k \in U(2) , \\
p_k(z, w) = w^{d_k} + \text{(lower order terms in $z, w$)}.
\end{gather*}
So $T_1$ is positively $J_1 \equiv (BR)_\#J$ holomorphic, and $T_2$ is positively $J_2 \equiv (B_2 R)_\# J$ holomorphic.  If we write $P_k = B_k \{ w = 0 \}$ then the tangent cone of $T_k$ at infinity is $m_k \llbracket P_k \rrbracket$, and (since $J_1 \neq J_2$) we have that $P_1 \neq P_2$.

Fix $x \in \spt\, T_1 \cap \spt\, T_2$, and let $\bC_1', \bC_2'$ be the tangent cones of $T_1, T_2$ at $x$, so that $\bC_1' + \bC_2'$ is the tangent cone of $T_1 + T_2$ at $x$.  We claim there are transverse planes $P_1', P_2'$, both positively $J'$-holomorphic for some orientation-preserving orthogonal complex structure, so that
\begin{equation}\label{eqn:main-2}
\bC_1' = m_1' \llbracket P_1' \rrbracket, \quad \bC_2' = m_2' \llbracket P_2' \rrbracket
\end{equation}
for some integers $m_1', m_2' > 0$.  In particular, $P_1', P_2'$ will have positive intersection number.  Also note \eqref{eqn:main-2} implies $T_1$ or $T_2$ can be singular at $x$ only if $x$ is a branch point.

We prove our claim.  First note that since $J_1 \neq J_2$ and both are orientation-reversing, there is no plane which is simultaneously $J_1$- and $J_2$-holomorphic.  So $\bC_1'$ is a sum of oriented planes with multiplicity, each being positive $J_1$-holmorphic but not $J_2$-holomorphic, and $\bC_2'$ is a sum of planes with multiplicity, each being positive $J_2$-holomorphic but not $J_1$-holomorphic.  On the other hand, by \cite{morgan:planes} every plane in both $\bC_1'$ and $\bC_2'$ must be positively $J'$-holomorphic with respect to some orthogonal complex structure $J'$.

We cannot have $J' = J_1$ or $J' = J_2$.  But then a given plane $P$ can only be positively holomorphic for both $J'$ and $J_i$ if $J'$ is obtained from $J_i$ by a real reflection in $P^\perp$, and in this case $P$ will be the \emph{unique} plane which is positively holomorphic with respect to both $J'$ and $J_i$.  Therefore $J'$ will be orientation-preserving, and $\bC_1'$, $\bC_2'$ will be a single plane $P_1'$, $P_2'$ (respectively) each with some multiplicity, as in \eqref{eqn:main-2}.

Since each $B_k R B_k^{-1} \spt\, T_k \equiv B_k \{ p_k(z, w) = 0 \}$ extends as an algebraic variety in $\CP^2$, we can find punctured Riemann surfaces $\Sigma_1, \Sigma_2$ (with compact closure), and simple branched immersions $v_k : \Sigma_k \to \C^2$ parameterizing $B_k \{ p_k(z, w) = 0 \}$ (with positive orientation).  Define $u_k = B_k R B_k^{-1} v_k$ so that $u_k(\Sigma_k) = \spt\, T_k$, and agrees with the orientation of $T_k$ on its regular part.

Choose a large ball $B_r$ so that for each $k = 1, 2$, we have $\sing\, T_k \subset B_{r/2}$ and $\spt\, T_k \cap \del B_\rho$ is $\eps \rho$-Hausdorff close to $P_k \cap \del B_\rho$ for all $\rho \geq r/2$.  Choose any $0 < \eps < \frac{1}{100} d_H(P_1 \cap B_1, P_2 \cap B_1)$ (where $d_H$ is the Hausdorff distance), so that $\spt\, T_1$, $\spt\, T_2$ are disjoint outside $B_{r/2}$.  Let $\pi_k$ be the orthogonal projection onto $P_k$, and define
\[
\tilde \Sigma_k = v_k^{-1}(\{ |\pi_k(y)| < 2r \}) \equiv u_k^{-1}(\{ |\pi_k(y)| < 2r \}),
\]
which we can assume (after changing $r$ if necessary) is a compact Riemann surface with smooth boundary.  Let $\eta : \R \to \R$ be a smooth non-increasing function which is $\equiv 1$ on $(-\infty, 1.1r]$ and $\equiv 0$ on $[1.9r, \infty)$, and define the truncated map
\[
\tilde u_k(x) = \eta(|(\pi_k\circ u_k)(x)|) u_k(x) + (1-\eta(|(\pi_k \circ u_k)(x)|)) (\pi_k \circ u_k)(x),
\]
similarly we can let $\tilde v_k(x) \equiv B_k R B_k^{-1} \tilde u_k(x)$ (recall that $B_k R B_k^{-1}$ fixes the plane $P_k$).  Then each $\tilde u_k, \tilde v_k : \tilde \Sigma_k \to B_r \subset \C^2$, and
\begin{gather}
\tilde u_k(\del \tilde \Sigma_k) = \tilde v_k(\del \tilde \Sigma_k) = \del B_{2r} \cap P_k , \label{eqn:main-2.5} \\
\tilde v_k(\tilde \Sigma_k) \cap B_r = v_k(\tilde \Sigma_k) \cap B_r, \quad \tilde u_k(\tilde \Sigma_k) \cap B_r = u_k(\tilde \Sigma_k) \cap B_r. \nonumber
\end{gather}
In particular, by our choice of $r$, we have $\spt\, T_1 \cap \spt\, T_2 = \tilde u_1(\tilde \Sigma_k) \cap \tilde u_2(\tilde \Sigma_2)$.

If $x \in \spt\, T_1 \cap \spt\, T_2$ is a branch point for $T_1$ (and hence $\tilde u_1$), then since the tangent planes of $\reg\, T_1$ vary continuously near $x$ (and approach $P_1'$ as they approach $x$), we can perturb $\tilde u_1$ nearby $x$ so that $\tilde u_1(\tilde \Sigma_1)$, $\tilde u_2(\tilde \Sigma_2)$ only intersect in their regular parts near $x$, the number of intersections only increases, and any new points of intersections will remain positive.  For example, if $v$ is any vector not in $P_1' \cup P_2'$, then we can consider a diffeomorphism like $\phi(y) = y + \eps v \xi(|y - x|)$ for $\xi$ a suitable bump function, and then take $\phi \circ \tilde u_1$ to be our perturbation.  Likewise the same argument would work if $x$ is a branch point of $T_2$.

Since the the collection of branch points is finite and contained in $B_{r/2}$, we can thus homotope $\tilde u_1$, $\tilde u_2$ a finite number of times as described above, keeping each $\tilde u_k$ fixed outside $B_r$, so that at any intersection point the maps $\tilde u_k$ are smooth, transverse immersions with positive intersection, and moreover $\# \spt\, T_1 \cap \spt\, T_2 \leq \# \tilde u_1(\tilde \Sigma_1) \cap \tilde u_2(\tilde \Sigma_2)$.

By intersection theory (e.g. \cite[Section 7]{micallef-white}) we obtain
\begin{equation}\label{eqn:main-3}
\# \spt\, T_1 \cap \spt\, T_2 \leq [\tilde u_1(\tilde \Sigma_1) ] \cdot [\tilde u_2(\tilde \Sigma_2)]
\end{equation}
where the right-hand side is a homotopy invariant for each $\tilde u_k$, among maps $\tilde\Sigma_k \to B_r$ which fix the boundary $\tilde u_k(\del \tilde \Sigma_k)$.  Using \eqref{eqn:main-2.5}, we can homotope $\tilde u_k$ to $\tilde v_k$ by e.g. linear interpolation
\[
(x, t) \mapsto (1-t) \tilde u_k(x) + t \tilde v_k(x),
\]
while will leave the boundary unchanged, and thereby deduce that the right-hand side of \eqref{eqn:main-3} also equals $[\tilde v_1(\tilde \Sigma_1)] \cdot [\tilde v_2(\tilde \Sigma_2)]$.  But now
\[
\tilde v_1(\tilde \Sigma_1) \cap \tilde v_2(\tilde \Sigma_2) = v_1(\Sigma_1) \cap v_2(\Sigma_2) = B_1 \{ p_1(z, w) = 0 \} \cap B_2 \{ p_2(z, w) = 0 \}
\]
and since each $v_k$ is a simple branched immersion parameterizing the algebraic variety $B_k\{ p_k(z, w) = 0 \}$, we can use Bezout to get
\[
[\tilde v_1(\tilde \Sigma_1) ] \cdot [\tilde v_2(\tilde \Sigma_2)] = d_1 d_2.
\] 
This completes Case (2).

\emph{Case (3).}  As before, for ease of notation we will simply write $i = 1$ and $j = 2$ in the below. Up to rigid motion, our setup becomes
\[
T_1 = m_1 \llbracket \{ p_1(z,w)=0\}\rrbracket;\quad T_2 = m_2 R_\#\llbracket \{p_2(z,w)=0\}\rrbracket;
\]
\[
\bC_2 = m_2 d_2 R_\#\llbracket \{w=0\}\rrbracket = m_2 d_2 \llbracket \{w=0\}\rrbracket, 
\]
where $\bC_2$ is the tangent cone at infinity of $T_2$.  So, $T_1$ is $J$-holomorphic and $T_2$ is $\overline{J}$-holomorphic where we write $\overline{J} = R_\# J$.

Unlike Cases (1), (2), there doesn't (as far as we can tell) need to be any positivity of intersection between $T_1, T_2$, but instead there is a different kind of a rigidity of the intersection set.

Fix any $x \in \spt\, T_1 \cap \spt\, T_2$, and let $\bC_1', \bC_2'$ be the tangent cones of $T_1$, $T_2$ at $x$ (respectively), so $\bC_1' + \bC_2'$ is the tangent cone of $T_1 + T_2$ at $x$.  We claim at least one of $\bC_1'$, $\bC_2'$ must be supported in $\{ w = 0 \}$.  Otherwise, suppose $\bC_1'$ contains an oriented plane $P_1'$ (with multiplicity), and $\bC_2'$ contained an oriented plane $P_2'$ (with multiplicity), so that neither $P_1'$ nor $P_2'$ coincides with $\{ w = 0 \}$.

From \cite{morgan:planes}, both $P_1'$ and $P_2'$ must be positively $J'$-holomorphic for some orthogonal complex structure $J'$.  Since $\{ w = 0 \}$ is the only plane which is positively holomorphic with respect to both $J$ and $\overline{J}$, we know $J' \neq J$ and $J' \neq \overline{J}$.  Therefore, we must have that $J'$ is obtained from $J$ by a real reflection about $P_1'$, and obtained from $\overline{J}$ from a real reflection about $P_2'$.  But $J$ is orientation-preserving while $\overline{J}$ is orientation-reversing, which means $J'$ must be both orientation-preserving and orientation-reversing, a contradiction.

From our claim we get
\begin{align*}
    \spt\,T_1\cap \spt\,T_2 & \subset \{(z,w):p_1(z,w) = (\del_zp_1)(z,w)=0\}\\
    & \qquad \cup\{(z,\overline{w}):p_2(z,w) = (\del_zp_2)(z,w) = 0\}, 
\end{align*}
and hence by Bezout we get
\[
\# \spt\, T_1 \cap \spt\, T_2 \leq d_1(d_1 - 1) + d_2(d_2 - 1).
\]
This completes Case (3).

\vspace{3mm}

Combining Cases (1), (2), (3), with the bound \eqref{eqn:main-concl3} we obtain
\begin{align*}
\# \sing\, T 
&\leq \frac{1}{2} \sum_{i, j \leq \ell, i\neq j} d_i d_j + \frac{1}{2} \sum_{i, j \geq \ell+1, i \neq j} d_i d_j + \frac{1}{2} \sum_i (d_i - 1)(d_i - 2) \\
&\quad\quad + \sum_{i \leq \ell < j} d_i(d_i + 1) + d_j(d_j+1) \\
&\leq \frac{1}{2} \sum_{i, j \leq \ell} d_i d_j + \frac{1}{2} \sum_{i, j \geq \ell+1} d_i d_j + (k-\ell) \sum_{i \leq \ell} d_i^2 + \ell \sum_{j \geq \ell+1} d_j^2 \\
&\leq k \sum_i d_i^2 \leq \Theta^3. \qedhere
\end{align*}

\end{proof}

\begin{proof}[Proof of Corollary \ref{cor:main}]
If $T$ is not already conical, then by the monotonicity formula $\Theta_T(x) < \Theta_T(\infty) = 2$ for all $x$, and hence $\Theta_T(x) = 1$ for every $x \in \spt\, T$.  So unless $T$ is a multiplicity two affine plane, by Allard's theorem $T$ is entirely regular and multiplicity one. The result then follows from Theorem \ref{thm:main}.
\end{proof}

\bibliography{refs}
\bibliographystyle{amsplain}

\end{document}